\documentclass[12pt]{amsart}

\usepackage{amsmath,amsfonts,amssymb,mathabx,latexsym,mathtools}
\usepackage{enumitem}
\usepackage[usenames, dvipsnames]{xcolor}
\usepackage{hyperref}
\usepackage{ytableau}
\usepackage{centernot}
\usepackage{tikz-cd}
\usepackage{tikz}
\usetikzlibrary{calc, shapes, backgrounds,arrows,positioning,plotmarks}
\tikzset{>=stealth',
  head/.style = {fill = white, text=black},
  plaque/.style = {draw, rectangle, minimum size = 10mm}, 
  pil/.style={->,thick},
  junct/.style = {draw,circle,inner sep=0.5pt,outer sep=0pt, fill=black}
  }

\newtheorem{theorem}{Theorem}[section]
\newtheorem{lemma}[theorem]{Lemma}
\newtheorem{proposition}[theorem]{Proposition}

\theoremstyle{definition}
\newtheorem{definition}[theorem]{Definition}
\newtheorem{examplex}[theorem]{Example}
\newenvironment{example}
  {\pushQED{\qed}\examplex}
  {\popQED\endexamplex}

\theoremstyle{remark}
\newtheorem{remark}[theorem]{Remark}

\numberwithin{equation}{section}

\newcommand{\SYT}{\ensuremath{\mathrm{SYT}}}
\newcommand{\gen}{\ensuremath{\mathrm{Gen}}}
\newcommand{\qygen}{\ensuremath{\mathrm{QYGen}}}
\newcommand{\incgl}{\ensuremath{\mathrm{Inc_{gl}}}}

\newcommand{\xx}{\ensuremath{\mathbf{x}}}

\newcommand{\wt}{\ensuremath{\mathrm{wt}}}
\newcommand{\des}{\ensuremath{\mathrm{Des}}}
\newcommand{\reg}{\ensuremath{\mathrm{reg}}}

\newcommand{\qsym}{\ensuremath{\mathrm{QSym}}}

\begin{document}


\title{The genomic Schur function is fundamental-positive}  

\author[O. Pechenik]{Oliver Pechenik}
\address[OP]{Department of Mathematics, University of Michigan, Ann Arbor, MI 48109}
\email{pechenik@umich.edu}


\date{\today}



\begin{abstract}
In work with A.~Yong, the author introduced genomic tableaux to prove the first positive combinatorial rule for the Littlewood-Richardson coefficients in  torus-equivariant $K$-theory of Grassmannians. We then studied the genomic Schur function $U_\lambda$, a generating function for such tableaux, showing that it is non-trivially a symmetric function, although generally not Schur-positive. Here we show that $U_\lambda$ is, however, positive in the basis of fundamental quasisymmetric functions. We give a positive combinatorial formula for this expansion in terms of gapless increasing tableaux; this is, moreover, the first finite expression for $U_\lambda$. Combined with work of A.~Garsia and J.~Remmel, this yields a compact combinatorial (but necessarily non-positive) formula for the Schur expansion of $U_\lambda$. 
\end{abstract}

\maketitle

%
\section{Introduction}
%
\label{sec:introduction}

The {\bf Grassmannian} $X = {\rm Gr}_k(\mathbb{C}^n)$ is the parameter space of $k$-dimensional vector subspaces of $\mathbb{C}^n$. The Grassmannian has the structure of a cell complex, where the cells are naturally indexed by partitions $\lambda$ whose Young diagrams fit inside a $k \times (n-k)$ rectangle. This cell complex structure yields the Schubert basis $\{ \sigma_\lambda \}$ of the integral cohomology ring $H^\star(X; \mathbb{Z})$. With respect to the Schubert basis, the structure constants of this $\mathbb{Z}$-algebra are the famous Littlewood-Richardson coefficients $c_{\lambda, \mu}^\nu \in \mathbb{Z}_{\geq 0}$; that is, 
\[
\sigma_\lambda \cdot \sigma_\mu = \sum_{\nu \subseteq k \times (n-k)} c_{\lambda, \mu}^\nu \sigma_\nu. 
\]
Classically, many positive combinatorial rules for these coefficients are known.

The Grassmannian also carries a natural action of a rank $n$ algebraic torus $T$. Considering the structure sheaves of the cell closures gives rise to an analogous Schubert basis $\{ \mathcal{O}_\lambda \}$ of the torus-equivariant $K$-theory ring $K_T(X)$ of equivariant algebraic vector bundles over $X$. The analogues of Littlewood-Richardson coefficients in this setting are the Laurent polynomials $K_{\lambda, \mu}^\nu  \in K_T({\rm pt})$ given by 
\[
\mathcal{O}_\lambda \cdot \mathcal{O}_\mu = \sum_{\nu \subseteq k \times (n-k)} K_{\lambda, \mu}^\nu \mathcal{O}_\nu. 
\]
Positive combinatorial rules for $K_{\lambda, \mu}^\nu$ were conjectured by A.~Knutson--R.~Vakil \cite{Coskun.Vakil} and by H.~Thomas--A.~Yong \cite{Thomas.Yong:HT}.

Genomic tableaux were introduced in \cite{Pechenik.Yong:KT} as the key object in the first proved positive combinatorial rule for $K_{\lambda, \mu}^\nu$. Genomic tableaux have been further studied in \cite{Monical,Gillespie.Levinson, Pechenik.Yong:KT2, Pechenik.Yong:genomic, Pylyavskyy.Yang}, yielding, in particular, resolutions of the Knutson-Vakil and Thomas-Yong conjectures mentioned above. In \cite{Pechenik.Yong:genomic}, A.~Yong and the author further developed the combinatorial theory of genomic tableaux with relation to the non-equivariant $K$-theory of $X$, introducing the \emph{genomic Schur function} as a natural deformation of the ordinary Schur function and generating function for genomic tableaux. This paper further studies the genomic Schur function, giving the first finite formulas for it, as well as proofs of new positivity properties that it enjoys.

Consider the Young diagram of a partition $\lambda$ (in English orientation, so the longer rows are above). A {\bf semistandard tableau} of shape $\lambda$ is a filling $T$ of the boxes of this Young diagram by positive integers that weakly increase from left to right along rows and strictly increase from top to bottom down columns. 

\begin{definition}[\cite{Pechenik.Yong:KT}]
A {\bf genomic tableau} is a semistandard tableau $T$ together with, for each $i \in \mathbb{Z}_{> 0}$, a decomposition of the boxes labeled $i$ into blocks, called {\bf genes}. These decompositions satisfy:
\begin{itemize}
\item if a gene $\mathcal{G}$ contains boxes labeled $i$ in columns $c_1, c_3$ with $c_1 < c_3$, then any $i$ in an intervening column $c_2$ with $c_1 < c_2 < c_3$ is in $\mathcal{G}$;
\item each gene $\mathcal{G}$ contains at most one box in any row.
\end{itemize}
\end{definition}

\begin{example}
The semistandard tableau 
\[
T = \ytableaushort{{*(Lavender) 1} {*(SkyBlue) 2} {*(Dandelion) 2},{*(SkyBlue)2}{*(SpringGreen)3}{*(Mulberry)4},{*(Mulberry)4}}
\]
is a genomic tableau with respect to the gene decomposition indicated by the coloring of the boxes. However, the decompositions
\[
\ytableaushort{{*(Lavender) 1} {*(SkyBlue) 2} {*(Dandelion) 2},{*(Dandelion)2}{*(SpringGreen)3}{*(Mulberry)4},{*(Mulberry)4}} \; \text{ and } \; \ytableaushort{{*(Lavender) 1} {*(SkyBlue) 2} {*(SkyBlue) 2},{*(Dandelion)2}{*(SpringGreen)3}{*(Mulberry)4},{*(Mulberry)4}}
\]
are not valid genomic tableaux; the first decomposition violates the first condition, while the second violates the second.
\end{example}

We write $\gen(\lambda)$ to denote the set of all genomic tableaux of shape $\lambda$. For $n >0$, $\gen^n(\lambda) \subset \gen(\lambda)$ denotes the subset of tableaux with all entries at most $n$.

If a gene $\mathcal{G}$ consists of boxes labeled $i$, we call $\mathcal{G}$ an {\bf $i$-gene}. A {\bf (strong) composition} is a finite sequence of positive integers, while a {\bf weak composition} is a finite sequence of nonnegative integers.  The {\bf weight} of a genomic tableau $T$ is the weak composition $\wt(T)$ whose $i$th component records the number of $i$-genes in $T$.
Just as Schur functions are generating functions for semistandard tableaux, genomic Schur functions are analogous generating functions for genomic tableaux. For a weak composition $a$, we use the shorthand $\xx^a \coloneqq \prod_{i \in \mathbb{Z}_{>0}} x_i^{a_i}$.

\begin{definition}[\cite{Pechenik.Yong:genomic}]\label{def:genomicSchur}
For a partition $\lambda$, the {\bf genomic Schur function} $U_\lambda$ is given by
\begin{equation}
U_\lambda \coloneqq \sum_{T \in \gen(\lambda)} \xx^{\wt(T)}.\label{eq:genomic_schur}
\end{equation}
The {\bf genomic Schur polynomial} $U_\lambda(x_1, \dots, x_n)$ in $n$ variables is given by specializing $U_\lambda$ by $x_m \to 1$ for $m > n$, or equivalently by restricting the sum in Equation~\eqref{eq:genomic_schur} to $\gen^n(\lambda)$.
\end{definition}

The power series $U_\lambda$ was shown in \cite[Theorem~6.6]{Pechenik.Yong:genomic} to be a symmetric function by a Bender-Knuth-type argument. Schur functions form a basis of the algebra of symmetric functions, so every symmetric function may be written uniquely as a linear combination of Schur functions. For many important symmetric functions appearing in algebraic combinatorics, these linear combinations have (or are conjectured to have) only \emph{positive} coefficients. Under standard interpretations of Schur functions, this yields immediate interpretations of such Schur-positive symmetric functions as representing Grassmannian cohomology classes or general linear group representations. Although $U_\lambda$ is Schur-positive for very small $\lambda$, it is not Schur-positive in general.

\begin{example}[{cf.~\cite[Example 6.7]{Pechenik.Yong:genomic}}]\label{ex:nonpositive}
Let $\lambda = (3,3,3)$. Then,
\begin{align*}
U_\lambda &= s_{333} + s_{22211} - s_{2222} + s_{3221} + s_{1111111} + 2s_{211111} + s_{22111} + s_{2221} + s_{31111} \\
&+ 4s_{111111} + 2s_{21111} + s_{11111}.
\end{align*}
We will explain this example further in Section~\ref{sec:Schur}.
\end{example}

Although genomic Schur functions are not Schur-positive, we will prove a weaker positivity property that they enjoy. 
A multivariate power series $f$ is {\bf quasisymmetric} if, for any two increasing sequences $i_1 < \dots < i_n$ and $j_1 < \dots < j_n$ and any sequence $e_1, \dots, e_n$ of nonnegative integers, the coefficients of the monomials $\prod_{k = 1}^n x_{i_k}^{e_k}$ and $\prod_{k = 1}^n x_{j_k}^{e_k}$ in $f$ are equal. Clearly, all symmetric functions (in particular, genomic Schur functions) are quasisymmetric. The collection of all quasisymmetric functions forms a $\mathbb{Z}$-algebra $\qsym$ with symmetric functions as a subalgebra.

The {\bf positive part} of a weak composition $a$ is the composition $a^+$ given by removing all the $0$'s in $a$. 
For two strong compositions $\alpha$ and $\beta$, we say $\beta$ {\bf refines} $\alpha$  (in symbols, $\beta \vDash \alpha$) if $\alpha$ can be obtained by summing consecutive components of $\beta$. For example, \[(1,2,1,1) \vDash (3,2),\] while \[(1,2,1,1) \nvDash (2,3).\]
For each strong composition $\alpha$, the {\bf fundamental quasisymmetric function} $F_\alpha$ is given by 
\[
F_\alpha = \sum_{a^+ \vDash \alpha} \xx^a,
\]
where the sum is over weak compositions whose  positive part refines $\alpha$. These quasisymmetric functions were introduced by I.~Gessel \cite{Gessel} with application to $P$-partitions; they have since become ubiquitous in algebraic combinatorics. Importantly, the fundamental quasisymmetric functions form an additive basis of the $\mathbb{Z}$-algebra $\qsym$. In particular, every \emph{symmetric} function can be written uniquely as a linear combination of fundamental quasisymmetrics. I.~Gessel \cite{Gessel} showed that Schur functions are \emph{positive} combinations of fundamental quasisymmetrics; hence, a Schur-positive symmetric function is necessarily also fundamental-positive. Our main result is that, while genomic Schur functions are not generally Schur-positive, they are nonetheless always fundamental-positive.

\begin{theorem}\label{thm:main}
For any partition $\lambda$, the genomic Schur function $U_\lambda$ is a finite positive combination of fundamental quasisymmetric functions.
\end{theorem}

Classically, Schur-positive symmetric functions correspond, under the \emph{Frobenius character} map, to representations of symmetric groups. Under this map, the irreducible representations of symmetric groups are taken to the corresponding Schur functions. Analogously, one has a \emph{quasisymmetric Frobenius character} map, taking representations of the (type A) $0$-Hecke algebras $\mathcal{H}_i(0)$ to quasisymmetric functions \cite{Duchamp.Krob.Leclerc.Thibon,Krob.Thibon}. The irreducible $0$-Hecke modules $\mathcal{C}_\alpha$ are all $1$-dimensional \cite{Norton} and correspond, under this map, to the fundamental quasisymmetric functions $F_\alpha$.  
Therefore, every homogeneous fundamental-positive quasisymmetric function $f = \sum_\alpha c_\alpha F_\alpha$ is the quasisymmetric Frobenius character of a $\mathcal{H}_i(0)$-module $M$; in particular, one may construct $M$ as the  direct sum of these $1$-dimensional irreducibles, one for each $F_\alpha$ in the expansion of $f$, i.e.\ $M = \bigoplus_\alpha \mathcal{C}_\alpha$. 

In this way, Theorem~\ref{thm:main} may be interpreted in terms of the representation theory of $0$-Hecke algebras. It would, however, be significantly more interesting to realize (the $k$th homogeneous piece of) $U_\lambda$ as the quasisymmetric Frobenius character of an \emph{indecomposable} $0$-Hecke module. Analogous such constructions have been obtained recently for the the \emph{dual immaculate quasisymmetric functions} \cite{Berg.Bergeron.Saliola.Serrano.Zabrocki} and for certain \emph{quasisymmetric Schur functions} \cite{Tewari.vanWilligenburg}. Unfortunately, we were not able to build indecomposable $0$-Hecke modules corresponding to the homogeneous pieces of $U_\lambda$. We further consider this perspective in Remark~\ref{rem:projective}.

We prove Theorem~\ref{thm:main} by establishing two explicit positive combinatorial formulas for the fundamental quasisymmetric expansion of $U_\lambda$. Both formulas are analogues of known formulas for the fundamental quasisymmetric expansion of a Schur function $s_\lambda$. In Section~\ref{sec:qy}, we prove a genomic analogue (Theorem~\ref{thm:qy}) of the formula of S.~Assaf--D.~Searles \cite{Assaf.Searles} in terms of quasiYamanouchi semistandard tableaux. In Section~\ref{sec:inc}, we prove an analogue (Theorem~\ref{thm:inc}) of the formula of I.~Gessel \cite{Gessel} in terms of standard Young tableaux. An advantage of the former formula is that it restricts better to an efficient formula for the finite-variable truncation $U_\lambda(x_1, \dots, x_n)$, while an advantage of the latter formula is that it allows an alternate definition of genomic Schur functions without reference to genomic tableaux. Section~\ref{sec:Schur} builds on these results to give a (signed) combinatorial rule (Theorem~\ref{thm:schur}) for the Schur expansion of a genomic Schur function. In particular, Proposition~\ref{prop:2row} proves Schur-positivity of $U_\lambda$ in the special case that the partition $\lambda$ has at most $2$ rows.

\section{QuasiYamanouchi tableaux}\label{sec:qy}

S.~Assaf--D.~Searles introduced the following notion of a quasiYamanouchi tableau.

\begin{definition}[{\cite[Definition~2.4]{Assaf.Searles}}]
A semistandard tableau $T$ is {\bf quasiYamanouchi} if, for each integer $i > 1$ that appears in $T$, there is some instance of $i$ weakly west of some instance of $i-1$.  
\end{definition}

We will say that a genomic tableau $U$ is {\bf quasiYamanouchi} if its underlying semistandard tableau is. Let $\qygen(\lambda)$ denote the set of quasiYamanouchi genomic tableaux of shape $\lambda$. Note that for any fixed $\lambda$, $\qygen(\lambda)$ is necessarily a finite set.

\begin{theorem}\label{thm:qy}
For any partition $\lambda$, we have
\[
U_\lambda = \sum_{T \in \qygen(\lambda)} F_{\wt(T)}.
\]
Similarly, for any $n > 1$, 
\[
U_\lambda(x_1, \dots, x_n) = \sum_{T \in \qygen^n(\lambda)} F_{\wt(T)}.
\]
\end{theorem}

Note that Theorem~\ref{thm:main} is immediate from Theorem~\ref{thm:qy}.
To prove Theorem~\ref{thm:qy}, we first need a genomic analogue of a map from \cite[Definition~2.5]{Assaf.Searles}, which we will call {\bf regularization}. (Regularization is called \emph{destandardization} in \cite{Assaf.Searles}; however, we avoid the latter name, as we will later require a \emph{$K$-standardization} map that is confusingly unrelated.)
Let $T$ be a genomic tableau. The regularization $\reg(T)$ of $T$ is obtained as follows: For any $i > 1$, if every $i$ in $T$ if strictly East of every $i-1$ in $T$, replace every $i$ by $i-1$, maintaining the same decomposition of the boxes of $T$ into genes. Iterate this process until no further replacements can be made.

\begin{lemma}\label{lem:reg}
The map $\reg :  \gen(\lambda) \to \qygen(\lambda)$ is well-defined and surjective. The fixed points of $\reg$ are exactly the quasiYamanouchi tableaux.
\end{lemma}
\begin{proof}
Clearly, the algorithm preserves weakly increasing rows. Since we do not replace $i$ by $i-1$ if $i$ and $i-1$ appear in the same column, the algorithm preserves strictly increasing columns. The validity of the resulting gene decomposition is similarly clear. Since the algorithm clearly terminates (entries only decrease), and terminates only when the tableau satisfies the quasiYamanouchi condition, regularization is a well-defined map $\reg :  \gen(\lambda) \to \qygen(\lambda)$. Clearly, $\reg(T) = T$ for $T \in \qygen(\lambda)$, so the map is surjective. Since $\reg$ certainly has no other fixed points, the second sentence of the lemma follows as well.
\end{proof}

\begin{proof}[Proof of Theorem~\ref{thm:qy}]
We consider only the infinitely-variable case explicitly; the proof in the finite-variable case is exactly analogous.

We have 
\begin{align*}
U_\lambda &= \sum_{T \in \gen(\lambda)} \xx^{\wt(T)} \\
&= \sum_{U \in \qygen(\lambda)} \sum_{V \in \reg^{-1}(U)} \xx^{\wt(V)},
\end{align*} 
where the first equality is by Definition~\ref{def:genomicSchur} and the second is by Lemma~\ref{lem:reg}. 
Hence, it is enough to show that 
\[\sum_{V \in \reg^{-1}(U)} \xx^{\wt(V)} = F_{\wt(U)},\]
 for $U \in \qygen(\lambda)$.
 
Fix $U \in \qygen(\lambda)$ and note that $\wt(U)$ is a strong composition. If $V \in \reg^{-1}(U)$, then $\wt(V)^+ \vDash \wt(U)$, since we obtain $U$ from $V$ by repeatedly replacing every $i$ by $i-1$. Conversely, suppose $a^+ \vDash \wt(U)$. There is a unique $V \in \reg^{-1}(U)$ with $\wt(V) = a$, obtained as follows. Number the $1$-genes of $U$ from left to right, then the $2$-genes of $U$ from left to right, etc. In $V$, the first $a_1$ of these genes are $1$-genes, the next $a_2$ of these genes are $2$-genes, etc. By construction, $V \in \reg^{-1}(U)$ and $\wt(V) = a$; uniqueness is clear from the lack of choices.
\end{proof}

\section{Increasing tableaux}\label{sec:inc}

In this section, we give another formula for the fundamental quasisymmetric expansion of a genomic Schur function. An attractive feature of this new formula is that it gives an alternative definition of the genomic Schur function, which avoids mention of genomic tableaux.

An {\bf increasing tableau} is a semistandard tableau with strictly increasing rows. Increasing tableaux were perhaps first studied in their own right in \cite{Thomas.Yong:K}, although they appeared earlier in various contexts (e.g., \cite{Edelman.Greene,Jockusch.Propp.Shor}). As in \cite{Mandel.Pechenik}, we say an increasing tableau $T$ is {\bf gapless} if the set of numbers appearing in $T$ is an initial segment of $\mathbb{Z}_{>0}$. We write $\incgl(\lambda)$ for the set of all gapless increasing tableaux of shape $\lambda$. Note that, for any $\lambda$, $\incgl(\lambda)$ is a finite set.

Let $T$ be an increasing tableau. Following \cite{DPS}, we say that the integer $i$ is a {\bf descent} of $T$ if there is some instance of $i$ in a higher row of $T$ than some instance of $i+1$.

\begin{example}\label{ex:descents}
Let $T = \ytableaushort{125,346,578}$. Then, the descents of $T$ are $2$, $4$, $5$, and $6$.
\end{example}

\begin{definition}
Let $T$ be a gapless increasing tableau with entry set $\{1, 2, \dots, n\}$. Consider the word $12 \cdots n$ and insert a bar $\vert$ after each letter $i$ that is a descent of $T$. 
The {\bf descent composition} $\des(T)$ of $T$ is the strong composition whose components are the lengths of the segments between bars in this word (read left to right).
\end{definition}

\begin{example}\label{ex:descent_comp}
For the tableau $T$ of Example~\ref{ex:descents}, we have $\des(T) = (2,2,1,1,2)$, corresponding to the word $12 \vert 34 \vert 5 \vert 6 \vert 78$.
\end{example}

\begin{theorem}\label{thm:inc}
For any partition $\lambda$, we have
\[
U_\lambda = \sum_{T \in \incgl(\lambda)} F_{\des(T)}.
\]
\end{theorem}

\begin{proof}
We first recall the {\bf $K$-standardization} map $\Phi$ from \cite[\textsection 2]{Pechenik.Yong:genomic}. For $U \in \gen(\lambda)$, number the $1$-genes of $U$ from left to right, then the $2$-genes of $U$ from left to right, etc., as in the proof of Theorem~\ref{thm:qy}. The $K$-standarization $\Phi(U)$ of $U$ is the gapless increasing tableau with the positive integer $i$ in each cell of the $i$th gene in this ordering.

Observe that if $U$ is quasiYamanouchi, then $i$ is a descent of $\Phi(U)$ precisely if the $i$th and $(i+1)$st genes of $U$ contain distinct labels. Hence, for $U \in \qygen(\lambda)$, we have $\wt(U) = \des(\Phi(U))$.

It is enough then to observe that the map $\Phi : \qygen(\lambda) \to \incgl(\lambda)$ is bijective. This fact follows from \cite[Theorem~2.4]{Pechenik.Yong:genomic}. Specifically, the inverse map is given as follows. A {\bf horizontal strip} is a skew partition with at most one box in any column. A {\bf Pieri strip} is an increasing tableau filling of a horizontal strip such that labels weakly increase from left to right. 
For a composition $\mu$, let 
\[
M_1(\mu) = \{1, 2, \dots, \mu_1 \}, M_2(\mu) = \{ \mu_1 + 1, \dots, \mu_2 \}, \text{ etc.}
\]
Say the gapless increasing tableau $V \in \incgl(\lambda)$ is {\bf $\mu$-Pieri-filled} if the numbers from $M_i(\mu)$ form a Pieri strip in $V$, for each $i$. Note that, in particular, every gapless increasing tableau $V \in \incgl(\lambda)$ is $\des(V)$-Pieri-filled. The {\bf $K$-semistandardization} $\Psi(V)$ of $V$ is then the genomic tableau given by placing the number $i$ in those boxes with entry from $M_i(\des(V))$ in $V$, declaring two boxes to be in the same gene if they have equal labels in $V$. It is easy to see that $\Psi(V) \in \qygen(\lambda)$, and that $\Phi$ and $\Psi$ are then mutually inverse injections between finite sets. The theorem follows.
\end{proof}

\begin{remark}
Theorem~\ref{thm:inc} could be taken as an alternative definition of the genomic Schur function $U_\lambda$, allowing one to define $U_\lambda$ without reference to genomic tableaux. Indeed, the origin of this paper was a question of Bruce Westbury to the author, asking if the power series of Theorem~\ref{thm:inc} was symmetric. Interestingly, it seems hard, however, to prove this symmetry without reference to genomic tableaux.
\end{remark}

\begin{remark}
Unlike Theorem~\ref{thm:qy}, Theorem~\ref{thm:inc} does not naturally lead to an efficient formula for the finite-variable genomic Schur polynomial $U_\lambda(x_1, \dots, x_n)$. To give the fundamental expansion of $U_\lambda(x_1, \dots, x_n)$ by increasing tableaux, one must consider all of the gapless increasing tableaux of shape $\lambda$, and then discard those whose descent composition has more than $n$ parts.
\end{remark}

\section{The Schur expansion of genomic Schur functions}\label{sec:Schur}

As shown in Example~\ref{ex:nonpositive} and \cite[Example~6.7]{Pechenik.Yong:genomic}, the genomic Schur function $U_\lambda$ is not always a positive sum of ordinary Schur functions. However, the genomic Schur function is Schur-positive for small values of $\lambda$, as shown in \cite[Table~1]{Pechenik.Yong:genomic}. In this section, we use the formula of Theorem~\ref{thm:inc} to shed some light on this phenomenon.

Building on the inverse Kostka matrix of \"O.~E\u{g}ecio\u{g}lu--J.~Remmel \cite{Egecioglu.Remmel}, E.~Egge--N.~Loehr--G.~Warrington  \cite{Egge.Loehr.Warrington} gave a combinatorial (but necessarily signed) formula for the Schur expansion of any symmetric function with known expansion into fundamental quasisymmetric functions. The situation of having a symmetric function with known fundamental expansion but unknown Schur expansion is not uncommon, especially in the context of the theory of Macdonald polynomials. In light of Theorem~\ref{thm:inc}, this is also the case with genomic Schur functions.

The Egge-Loehr-Warrington formula (or its relatives) can then sometimes be used to give combinatorial proofs of Schur-positivity, in combination, for example, with a sign-reversing involution; see, for instance, work of E.~Sergel \cite{Sergel} and of D.~Qiu--J.~Remmel \cite{Qiu.Remmel}. Such an approach is certainly not possible for general genomic Schur functions, since genomic Schur functions are not Schur-positive; however, we still obtain a reasonably compact (but cancellative) formula for the Schur expansion. We will, moreover, obtain a \emph{positive} formula for the Schur expansion of $U_\lambda$ in the special case that $\lambda$ has at most $2$ rows.

Instead of the Egge-Loehr-Warrington formula, we will use the somewhat simpler (but essentially equivalent) rule by A.~Garsia--J.~Remmel \cite{Garsia.Remmel}. An alternative proof of this later rule has been given by I.~Gessel \cite{Gessel:expansion}. We first recall the notion of Schur functions indexed by general compositions. For $\alpha$ a composition, the Schur function $s_\alpha$ is defined to be the determinant $| h_{\alpha_i - i + j} |$, where $h_k$ denotes the $k$th complete homogeneous symmetric function. Note that this definition restricts to the ordinary Jacobi-Trudi formula for ordinary Schur functions in the case that $\alpha$ is a partition. More generally, we have, for each composition $\alpha$, that either $s_\alpha = 0$ or else $s_\alpha = \pm s_\lambda$ for some $\lambda$. The partition $\lambda$ in question may be found easily by iterating the straightening law 
\begin{equation}\label{eq:straightening}
s_{(\alpha_1, \dots, \alpha_i, \alpha_{i+1}, \dots \alpha_k)} = - s_{(\alpha_1, \dots, \alpha_{i+1} - 1, \alpha_i + 1, \dots, \alpha_k)},
\end{equation}
for $\alpha_{i+1} \neq 0$.
 Further discussion may be found in \cite{Macdonald}.

\begin{theorem}\label{thm:schur}
For any partition $\lambda$, we have
\[
U_\lambda= \sum_{T \in \incgl(\lambda)} s_{\des(T)}.
\]
\end{theorem}
\begin{proof}
This is immediate by combining Theorem~\ref{thm:inc} and \cite[Theorem~1]{Garsia.Remmel}.
\end{proof}

\begin{example}\label{ex:33}
We may see that $U_{33}$ is Schur-positive as follows. By Theorem~\ref{thm:schur}, we consider the $11$ gapless increasing tableaux
\[\begin{array}{ccccc}
\ytableaushort{12 {*(SkyBlue)3},456} & \ytableaushort{{*(SkyBlue)1}{*(SkyBlue)3}{*(SkyBlue)5},246} & \ytableaushort{{*(SkyBlue)1}3{*(SkyBlue)4},256} & \ytableaushort{1{*(SkyBlue)2}{*(SkyBlue)5},346} & \ytableaushort{1{*(SkyBlue)2}{*(SkyBlue)4},356} \\ \\
\ytableaushort{{*(SkyBlue)1}{*(SkyBlue)2}{*(SkyBlue)4},{*(SkyBlue)2}35} & \ytableaushort{{*(SkyBlue)1}2{*(SkyBlue)3},245} & \ytableaushort{1{*(SkyBlue)2}{*(SkyBlue)4},3{*(SkyBlue)4}5} & \ytableaushort{{*(SkyBlue)1}{*(SkyBlue)3}{*(SkyBlue)4},2{*(SkyBlue)4}5} & \ytableaushort{1{*(SkyBlue)2}{*(SkyBlue)3},{*(SkyBlue)3}45} \\ \\
& & \ytableaushort{{*(SkyBlue)1}{*(SkyBlue)2}{*(SkyBlue)3},{*(SkyBlue)2}{*(SkyBlue)3}4} & &
\end{array}\]
of shape $(3,3)$, where the descents are shaded in \textcolor{SkyBlue}{blue}.
Hence, we have
\begin{align*}
U_{33} &= s_{33} + s_{1221} + s_{132} + s_{231} + s_{222} \\
&+ s_{1121} + s_{122} + s_{221} + s_{1211} + s_{212} \\
&+ s_{1111}.
\end{align*}
Applying the straightening law~\eqref{eq:straightening}, this becomes
\begin{align*}
U_{33} &= s_{33} + 0 - s_{222} + 0 + s_{222} \\
&+ 0 + 0 + s_{221} + 0 + 0 \\
&+ s_{1111} \\
&= s_{33} + s_{221} + s_{1111},
\end{align*}
establishing the Schur-positivity of $U_{33}$.
\end{example}

The Schur-positivity in Example~\ref{ex:33} is a special case of a more general phenomenon.
\begin{proposition}\label{prop:2row}
If $\lambda = (m,n)$ is a partition with two rows, then $U_\lambda$ is Schur-positive. 

More specifically,
if $m=n$, we have that $U_{(m,m)}$ is a multiplicity-free sum of flag-shaped Schur functions:
\[
U_{(m,m)} = \sum_{k=0}^{m-1} s_{(m-k,m-k,1^k)},
\]
where $1^k$ denotes a sequence of $k$ $1$'s. 
If $m>n$, then
 $U_{(m,n)}$ is the multiplicity-free sum:
\[
U_{(m,n)} = s_{(m,n)} +  \sum_{k=1}^{n-1} s_{(m-k,n-k,1^k)} + s_{(m-k,n-k+1,1^{k-1})} .
\]
\end{proposition} 
\begin{proof}
By \cite[Proposition~2.1]{Pechenik}, there is a bijection between gapless increasing tableaux of shape $(m,m)$ with maximum entry $2m-k$ and standard Young tableaux of shape $(m-k,m-k, 1^k)$. By \cite[Proof of Corollary~2.2]{Pechenik}, this bijection preserves descent sets. Thus, writing $\SYT(\lambda)$ for the set of all standard Young tableaux of shape $\lambda$, we have
\begin{align*}
U_{(m,m)} &= \sum_{T \in \incgl(m,m)} F_{\des(T)} \\
&= \sum_{k=0}^{m-1} \sum_{S \in \SYT(m-k,m-k,1^k)} F_{\des(S)},
\end{align*}
by combining these facts with Theorem~\ref{thm:inc}.
The rectangular case of the proposition then follows from I.~Gessel's formula \cite{Gessel} for the fundamental expansion of an ordinary Schur function.

To prove the case $m>n$, it suffices to observe that the bijection of \cite[Proposition~2.1]{Pechenik} easily extends to a bijection between gapless increasing tableaux of shape $(m,n)$ with maximum entry $m+n-k$ and standard Young tableaux whose shape is either $(m-k,n-k,1^k)$ or $(m-k, n-k+1, 1^{k-1})$. Since this extended bijection also preserves descent sets, the $m>n$ case then follows analogously to the rectangular case.
\end{proof}
\begin{remark}\label{rem:projective}
In the Schur-positive situations of Example~\ref{ex:33} and Proposition~\ref{prop:2row}, one might hope to realize the homogeneous pieces of $U_\lambda$ as quasisymmetric Frobenius characters of \emph{projective} $0$-Hecke representations. This is not, however, possible. Using the explicit characterization of all indecomposable projective $0$-Hecke representations described in \cite{Huang}, it is straightforward to check, for example, that $s_{221}$ (the degree $5$ piece of $U_{33}$, as identified in Example~\ref{ex:33})  is not the quasisymmetric Frobenius character of any projective $0$-Hecke module.
\end{remark}

\begin{example}\label{ex:333}
To understand the failure of Schur-positivity in Example~\ref{ex:nonpositive}, it is enough to consider the $197$ gapless increasing tableaux of shape $(3,3,3)$. In fact, since the Schur-nonpositivity is restricted to homogeneous degree $8$, we need only consider the $84$ of these tableaux with maximum entry $8$. From the descent sets of these, we obtain by Theorem~\ref{thm:schur} that the homogeneous degree $8$ piece of $U_{333}$ is 
\begin{align*}
s_{1111211} &+ 2s_{111212} + 2s_{111221} + s_{111311} + s_{1121111} + s_{112112} + 3s_{112121} + s_{11213}  \\
&+ 2s_{112211} + 3s_{11222} + s_{11231} + s_{113111} + s_{11312} + s_{11321} + s_{121112} \\
&+ 2s_{121121} + s_{12113} + 3s_{121211} + 3s_{12122} + 2s_{12131} + 2s_{122111} + 3s_{12212} \\
&+ 4s_{12221} + s_{1223} + s_{12311} + s_{1232} + s_{13112} + 2s_{13121} + s_{13211} + s_{1322} \\
&+ s_{211121} + s_{211211} + 2s_{21122} + s_{21131} + 2s_{212111} + 2s_{21212} + 3s_{21221} \\
&+ s_{2123} + s_{21311} + s_{2132} + 2s_{22112} + 3s_{22121} + s_{2213} +3s_{22211} + 2s_{2222} \\
&+ s_{2231} + s_{2312} + s_{2321} + s_{31121} + s_{31211} + s_{3122} + s_{3212} + s_{3221}.
\end{align*}
Deleting the terms that are $0$, we obtain that the degree $8$ piece of $U_{333}$ is
\[
s_{13211} + s_{1322} + s_{21311} + s_{2132} + s_{2213} + 3s_{22211} +2s_{2222} + s_{3221}.
\]
Applying the straightening law~\eqref{eq:straightening}, this becomes 
\begin{align*}
&-s_{22211} - s_{2222} - s_{22211} - s_{2222} - s_{2222} + 3s_{22211} + 2s_{2222} + s_{3221} \\
&= s_{22211} - s_{2222} + s_{3221},
\end{align*}
as given in Example~\ref{ex:nonpositive}.
\end{example}

\begin{remark}
In light of Theorems~\ref{thm:inc} and~\ref{thm:schur}, it would be desirable to have enumerations of gapless increasing tableaux. For some special cases of partitions $\lambda$, such counting formulas appear in \cite{Pechenik,Pressey.Stokke.Visentin}, where they are related to small Schr\"oder paths and their higher-dimensional analogues. In general, however, there unfortunately does not appear to be a simple formula for the number of gapless increasing tableaux of shape $\lambda$.
\end{remark}

\section*{Acknowledgements}
This paper was inspired by conversations with Bruce Westbury during the conference ``SageDays@ICERM: Combinatorics and Representation Theory,'' held July 2018 at the Institute for Computational and Experimental Research in Mathematics. Thanks to the organizers (Gabriel Feinberg, Darij Grinberg, Ben Salisbury, and Travis Scrimshaw) for creating such a productive environment.
The author is also grateful for helpful conversations with Dominic Searles, Emily Sergel and David Speyer.

The author was supported by a Mathematical Sciences Postdoctoral Research Fellowship (\#1703696) from the National Science
Foundation.

%
%

\bibliographystyle{amsalpha} 
\bibliography{genomicSchur}

\end{document}